\documentclass[11pt]{article}

\usepackage{amsfonts, amsmath, amsthm, verbatim}
\usepackage{graphicx}
\usepackage{tikz}
\usepackage{epsfig}
\usetikzlibrary{backgrounds}

\newcommand{\Z}{\mathbb{Z}}
\newcommand{\N}{\mathbb{N}}
\newcommand{\R}{\mathbb{R}}
\renewcommand{\i}{\mathfrak{i}}
\renewcommand{\d}{\mathrm{d}}

\def\Var{{\mathbf{Var\,}}}

\def\E{\mathbb{E}}
\def\p{\mathbb{P}}

\theoremstyle{plain}
\newtheorem{lemma}{Lemma}
\newtheorem{theorem}{Theorem}

\theoremstyle{remark}
\newtheorem{remark}{Remark}

\begin{document}

\title{Tail probabilities of St.~Petersburg sums, trimmed sums, and their
limit}

\author{
Istv\'an Berkes
\thanks{Institute of Statistics, Graz University of Technology,
M\"unzgrabenstra{\ss}e 11/III, A-8010 Graz, Austria
\texttt{berkes@tugraz.at}.
}
\and
L\'aszl\'o Gy\"orfi
\thanks{Department of Computer Science and Information Theory, Budapest University
of Technology and Economics, 1521 Stoczek u. 2, Budapest, Hungary,
\texttt{gyorfi@cs.bme.hu}.
} \and P\'eter Kevei
\thanks{Center for Mathematical Sciences, Technische Universit\"at M\"unchen,
Boltzmannstra{\ss}e 3, 85748 Garching, Germany,
\texttt{peter.kevei@tum.de}.
} }

\maketitle

\begin{abstract}
We provide exact asymptotics for the tail probabilities $\p \{ S_{n,r} > x \}$ 
as $x \to \infty$, for fix $n$, where $S_{n,r}$ is the $r$-trimmed
partial sum of i.i.d.~St.~Petersburg random variables. In particular, we prove that
although the St.~Petersburg distribution is only
O-subexponential, the subexponential property almost holds.
We also determine the exact tail behavior of the $r$-trimmed limits.

\noindent 
\textit{Keywords:} St.~Petersburg sum; Trimmed sum; Tail asymptotic; Semistable law \\
\noindent \textit{MSC 60F05, MSC 60E07}
\end{abstract}

\section{Introduction} \label{sect:intro}

Peter offers to let Paul toss a fair coin repeatedly
until it lands heads and pays him $2^{k}$ ducats if this
happens on the $k^{\text{th}}$ toss, where $k\in\N = \{1,2,\ldots\}$.
This is the so-called classical St.~Petersburg
game. If $X$ denotes Paul's winning, then
$\p\left\{ X = 2^{k} \right\}  = 2^{-k}$, $k\in\N$.
Put $\lfloor x \rfloor$ for the lower integer part,
$\lceil x \rceil$ for the upper integer part,  and
$\{ x \}$ for the fractional part  of $x$. Then
the distribution function of the gain is
\begin{equation} \label{eq:dist-stp}
F   (x)  = \p\left\{ X \leq x \right\} = \left\{
\begin{array}{ll}
0, & x < 2 \,, \\
1- \frac{1}{2^{\lfloor \log_2 x \rfloor }} =
1- \frac{ 2^{ \{ \log_2 x \} } }{x},
&  x \geq 2 \,,
\end{array}
\right.
\end{equation}
and its quantile function $F^{-1}(s) = Q(s) = \inf \{ x : \, s \leq F(x) \}$ is
\begin{equation} \label{eq:quant}
Q(s) =
\begin{cases}
2, & s = 0, \\
2^{\lceil - \log_2 (1-s) \rceil} = \frac{2^{\{ \log_2 (1-s) \}}}{1-s}, & s \in (0,1).
\end{cases}
\end{equation}

Let $X_1, X_2, \ldots$ be i.i.d.~St.~Petersburg random variables,
and let
\[
S_n = X_1 + \ldots + X_n \quad \text{and} \quad
X_n^* = \max_{1 \leq i \leq n} X_i
\]
denote their partial sum and their maximum, respectively.
To define the $r$-trimmed sum, let $X_{1n} \geq X_{2n} \geq \ldots \geq X_{nn}$
be the ordered sample of the variables $X_1, X_2, \ldots, X_n$. For $r \geq 0$
put
\[
S_{n,r} = \sum_{k=r+1}^n X_{kn},
\]
that is, from the partial sum we subtract the $r$ largest observations. Note that
$S_{n,0} = S_n$ is the St.~Petersburg sum, while $S_{n,1} = S_n - X_n^*$ is the 1-trimmed sum.

In order to state the necessary and sufficient condition for the existence of the limit,
we introduce the positional parameter
\[ 
\gamma_n = \frac{n}{2^{\lceil \log_2 n \rceil}} \in (1/2, 1],
\] 
which shows the position of $n$ between two consecutive powers of 2.
Since the function $2^{ \{ \log_2 x \} }$ in the numerator in (\ref{eq:dist-stp})
is not slowly varying at infinity, 
the St.~Petersburg distribution is not in the domain of attraction of any
stable law or max-stable law, so limit distribution neither for the centered and normed sum,
nor for the centered and normed maximum, holds true.
What holds instead for the sum is the merging theorem
\begin{equation} \label{eq:sum-merge}
\sup_{x \in \mathbb{R}} \left|
\p\left\{ \frac{S_n}{n} - \log_2 n \leq x \right\}  -
G_{\gamma_n}(x) \right| \to 0, \quad \text{as } n \to \infty,
\end{equation}
shown by Cs\"org\H{o} \cite{Cs02},
where $G_\gamma$ is the distribution function of the infinitely divisible random
variable $W_\gamma$, $\gamma \in (1/2, 1]$, with
characteristic function
\[
\E \left( e^{ \i t W_{\gamma}} \right)
= \exp \left( \i t\left[ s_\gamma
+ u_{\gamma} \right] + \int_0^{\infty} \left( e^{\i t x}
- 1 - \frac{\i t x  }{1+x^2} \right) \d R_{\gamma}(x) \right)
\]
with $s_\gamma= - \log_2 \gamma$, 
$u_\gamma = \sum_{k=1}^\infty \frac{\gamma^2}{\gamma^2 + 4^k} -
\sum_{k=0}^\infty \frac{1}{1 + \gamma^2 4^k}$,
and right-hand-side L\'evy function
\begin{equation} \label{eq:Levy-func}
R_{\gamma}(x) = -\,\frac{\gamma}{2^{\lfloor \log_2 (\gamma x) \rfloor}} =
- \frac{2^{\{ \log_2 (\gamma x) \} } }{x}\,, \quad x>0.
\end{equation}
Convergence of $S_n / n - \log_2 n$ along subsequences $\{ n_k = \lfloor \gamma 2^k \rfloor \}$, $\gamma \in (1/2,1]$, was
first shown by Martin-L\"of in 1985 \cite{ML}.

For the maximum we have
\begin{equation} \label{eq:max-pr-merge}
\sup_{j \in \Z}
\left|  \p\left\{ X_n^* = 2^{ \lceil \log_2 n \rceil + j} \right\}
 - p_{j, \gamma_n} \right|
= O(n^{-1}),
\end{equation}
in particular
$\p\left\{ X_n^* = 2^{ \lceil \log_2 n \rceil + j } \right\}
\sim  p_{j,\gamma_n}$, for any
$j \in \Z$, as $n \to \infty$, where
\[ 
p _{j,\gamma} = e^{- \gamma 2^{-j}} \left( 1- e^{- \gamma 2^{-j}}\right), \quad
j \in \Z, \ \gamma \in [1/2, 1].
\] 
See formula (4) by Berkes et al.~in \cite{BCsCs}, or Lemma 1
by Fukker et al.~in \cite{FuGyKe14} in the general case.

The limit theorems (\ref{eq:sum-merge}) and (\ref{eq:max-pr-merge}) suggest
that the irregular oscillating behavior is due to the maximum, which is 
also indicated by the following fact.
It is well-known
(see Chow and Robbins \cite{ChRo61} and Adler \cite{Adler}) that
\[
1 = \liminf_{n \to \infty} \frac{S_n}{n \log_2 n} <
\limsup _{n \to \infty} \frac{S_n}{n \log_2 n}=\infty \quad \text{a.s.,}
\]
while the trimmed sum has nicer behavior, concerning at least the almost
sure limits, since
\[
\lim_{n \to \infty} \frac{S_n - X_n^*}{n \log_2 n} = 1 \quad \text{a.s.}
\]
(cf. Cs\"org\H{o} and Simons \cite{CsSi96}).
For further results and history of St.~Petersburg games see Cs\"org\H{o}
\cite{Cs02} and the references therein.

As a continuation of our studies of the joint behavior of $S_n$ and
$X_n^*$ in \cite{FuGyKe14}, we investigate the properties of the trimmed sum
$S_{n,r}$ both for fix $n$ and for $n \to \infty$.
Figure \ref{fig1} shows the histograms of the St.~Petersburg sum and of the
1-trimmed St.~Petersburg sum.
One can see that the histogram of $\log_2 S_n$ is mixtures of unimodal densities
such that the first lobe is a mixture of overlapping densities,
while the side-lobes have disjoint support.
For the histogram of $\log_2 (S_{n}-X_n^*)$  the side-lobes almost disappear, so
the trimmed version has smaller tail.
According to Proposition 7 in [9], for large $X_n^*$ one gets
$S_n / X_n^* \approx 1$, or equivalently
$(S_n-X_n^*)/X_n^* \approx 0$, which explains the disappearance of side-lobes.

\begin{figure}
\begin{center}
\includegraphics[width=10cm]{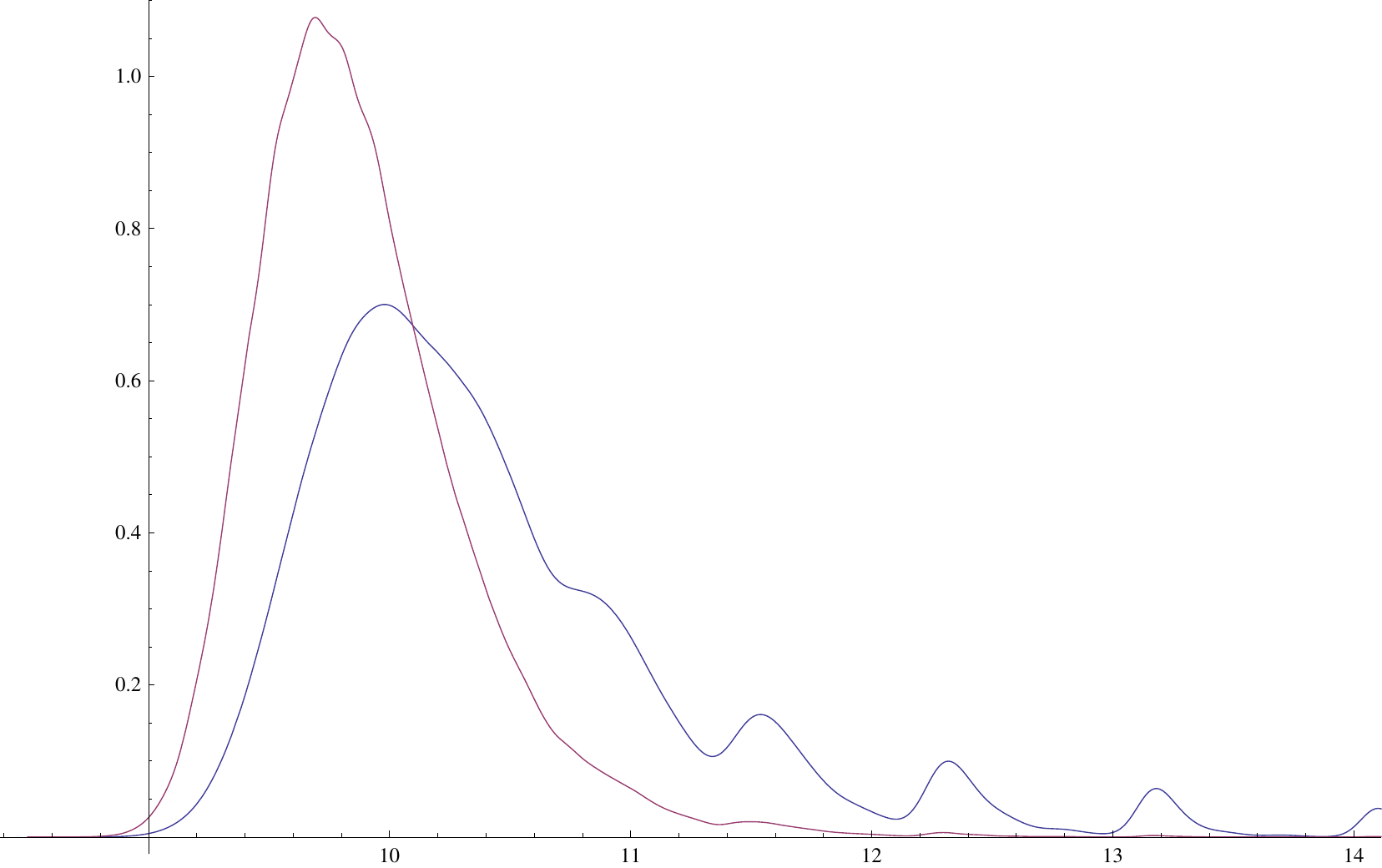}
\caption{The histograms of $\log_2 S_{n}$ and of $\log_2 (S_{n}-X_n^*)$ for
$n=2^{7}$.} \label{fig1}
\end{center}
\end{figure}

In Section \ref{sect:fixn} we investigate asymptotic behavior of the tail of the
distribution function of the sum $S_{n,r}$ for fix $n$. In Theorem \ref{th:Snr-tail}
we determine the exact tail behavior of $\p \{ S_{n,r} > x \}$.
In particular, we show that the St.~Petersburg distribution
is almost subexponential in a well-defined sense.

In Section \ref{sect:ntoinf} we let $n \to \infty$. 
In Theorem \ref{th:trimmed} we determine 
$\{ W_\gamma^* : \gamma \in (1/2, 1] \}$, the set of
the possible subsequential limit distributions of $(S_n - X_n^*)/n - \log_2 n$,
and we obtain  an infinite series representation for the distribution function.
This result was first obtained by Gut and Martin-L\"of in their Theorem 6.1 in \cite{GML}.
They investigate the so-called max-trimmed St.~Petersburg game, where from the sum
$S_n$ all the maximum values are subtracted. Theorem \ref{th:trimlim} states
the  limit theorem for the trimmed sums under arbitrary trimming, while in
Theorem \ref{th:trimtail} we determine the tail behavior of the limit.
In particular, we obtain exact tail asymptotics for a semistable law.
Finally, in Section \ref{sect:general} we mention some of these results
without proof in case of generalized St.~Petersburg games.

\section{Tail behavior of the sum and the trimmed sum} \label{sect:fixn}

In this section the number of summands $n$ is fix, and we are interested
in the tail behavior of $S_{n,r}$.

\subsection{The O-subexponentiality of the St.~Petersburg distribution}

First we summarize some basic facts on subexponential distributions.
Let $G$ be a distribution function of a non-negative random variable $Y$.
Put ${\overline G }(x) = 1 - G(x)$. The distribution $G$ is
\textit{subexponential}, $G \in \mathcal{S}$, if
\begin{equation} \label{eq:subexp-def}
\lim_{x \to \infty} \frac{{\overline {G * G}}(x)}{\overline G (x) } = 2,
\end{equation}
where $*$ stands for the usual convolution, and $G^{n*}$ is the $n^\text{th}$
convolution power, for $n \geq 2$. The characterizing property of the subexponential
distributions is that the sum of i.i.d.~random variables behaves like
the maximum of these variables, that is for any $n \geq 1$
\[
\lim_{x \to \infty} \frac{\p \{ Y_1 + \ldots + Y_n > x \}}
{ \p \{ \max \{ Y_i \, : \, i=1,2, \ldots, n \} > x \}} =1,
\]
or equivalently
\begin{equation} \label{eq:subexp-prop2}
\lim_{x \to \infty} \frac{\p \{ Y_1 + \ldots + Y_n > x \}}{\p \{ Y_1 > x \}} = n.
\end{equation}
For properties of subexponential distributions and their use in practice we
refer to the survey paper by Goldie and Kl\"uppelberg \cite{GoldieKlupp}.

It is well-known that distributions with regularly varying tails are
subexponential. What makes the St.~Petersburg game so interesting is that
its tail is not regularly varying. In fact
it was already noted by Goldie  \cite{Goldie} that the St.~Petersburg
distribution $F$ is not subexponential. What we have instead is that
\begin{equation} \label{eq:sumdf2}
2 = \liminf_{x \to \infty} \frac{{\overline {F  * F }}(x)}{\overline F  (x) }
< \limsup_{x \to \infty} \frac{{\overline {F  * F }}(x)}{\overline F  (x) } = 4.
\end{equation}
This can be proved by showing that for $1 \leq k \leq \ell$
\[
\p \{ X_1 + X_2 > 2^k + 2^{\ell} \} =
\begin{cases}
2 \cdot 2^{-\ell} + 2 \cdot 2^{-(\ell + k)} - 4 \cdot 2^{-2\ell}, & \text{for }
\ell > k, \\
2 \cdot 2^{-\ell} - 2^{- 2 \ell}, & \text{for } \ell = k,
\end{cases}
\]
from which
\[
\lim_{\ell \to \infty}
\frac{\p \{ X_1 + X_2 > 2^\ell \} }{\p \{ X_1 > 2^\ell \}} = 4,
\quad \text{ and } \
\lim_{\ell \to \infty}
\frac{\p \{ X_1 + X_2  > 2^\ell -1 \} }{\p \{ X_1 > 2^\ell -1 \}} = 2.
\]
Moreover, it is simple to see that 4 is in fact the limsup.

This naturally leads to the extension of subexponentiality. A distribution $G$
is \textit{O-subexponential}, $G \in \mathcal{OS}$, if
\[
l^*(G) := \limsup_{x \to \infty}  \frac{{\overline {G * G}}(x)}{\overline G (x) }
< \infty.
\]
It is known that the corresponding $\liminf$ is always greater than, or equal to $2$,
and it was shown recently by Foss and Korshunov \cite{FK} that it is exactly 2 for any
heavy-tailed distribution. The notion of O-subexponentiality was introduced
by Kl\"uppelberg \cite{K}. The properties of the $\mathcal{OS}$ class, in
particular when the distribution is also infinitely divisible,
were investigated by Shimura and Watanabe \cite{SW}. In their Proposition 2.4
they prove that if $G \in \mathcal{OS}$ then for every $\varepsilon > 0$ there
is a $c > 0$ such that for all $n$ and $x \geq 0$
\[
\frac{\overline {G^{n*}} (x)}{\overline G(x) } \leq c (l^*(G) - 1 + \varepsilon)^n.
\]
In the St.~Petersburg case $l^*(F )=4$. In Theorem \ref{th:Snr-tail} we determine the exact
asymptotic behavior of $\overline {F^{n*}}(x)$, which, in particular, implies a linear
bound in $n$ instead of the exponential.

Let us examine the case $n=2$ in detail. Note that
$\p \{ X_1 > 2^{\ell} \} = \p \{ X_1 > 2^{\ell} + 2^{k} \}$ for $k < \ell$, therefore
from (\ref{eq:sumdf2})
\[
\frac{ \p \{ X_1 + X_2 > 2^{\ell} + 2^{k} \}} { \p \{ X_1 > 2^{\ell} + 2^{k} \} } =
2 + 2 \cdot 2^{-k} - 4 \cdot 2^{-\ell}.
\]
From this it is clear that when both $\ell$ and $k$ tends to infinity, then the
limit exists and equal to 2; in particular for any $\delta > 0$
\[
\lim_{x \to \infty, \{ \log_2 x \} \geq \delta }
\frac{ \p \{ X_1 + X_2 > x \}} { \p \{ X_1 > x \} } = 2,
\]
where $\{ x \}$ stands for the fractional part of $x$.
That is, the St.~Petersburg distribution is `almost subexponential'.
We prove the corresponding result for general $n$, i.e.~for any $\delta > 0$
\[
\lim_{x \to \infty, \{ \log_2 x \} \geq \delta}
\frac{ \p \{ S_n > x \} } {\p \{ X_1 > x \} } = n.
\]

\smallskip

\begin{theorem} \label{th:Snr-tail}
For any $0 \leq r <  n $ we have as $x \to \infty$
\begin{equation} \label{eq:exacttail}
\begin{split}
\p \{ S_{n,r} > x \} & \sim 
\frac{2^{(r+1) \{\log_2 x \}}}{x^{r+1}} \binom{n}{r+1} \\
& \phantom{\sim} \, \times \left( 1 + \p \left\{ S_{n-r-1} >  x (1 - 2^{-\{ \log_2 x \}}) \right\} (2^{r+1} - 1)  \right).
\end{split}
\end{equation}
In particular, for any $0 < \delta < 1$,
\begin{equation} \label{eq:tail-as}
\lim_{x \to \infty, \{ \log_2 x \} > \delta }
\p \{ S_{n,r} > x \} \frac{x^{r+1}}{2^{(r+1) \{\log_2 x \}}} = \binom{n}{r+1}.
\end{equation} 
\end{theorem}

\begin{proof}
Let $X_{1n} \geq X_{2n} \geq \ldots \geq X_{nn}$ be the ordered sample of the
variables $X_1, X_2, \ldots, X_n$.
Using the well-known quantile representation, the form of the quantile function (\ref{eq:quant}),
and that $U \stackrel{\mathcal{D}}{=} 1 - U$, for $U \sim  \mathrm{Uniform}(0,1)$,
we obtain
\begin{equation} \label{eq:quantile-uni}
\left( X_{1n}, \ldots, X_{nn} \right)
\stackrel{\mathcal{D}}{=}
\left( 2^{\lceil \log_2 U_{1n}^{-1}\rceil}, \ldots,  2^{\lceil \log_2 U_{nn}^{-1}\rceil} \right),
\end{equation}
where $U_{1n} \leq U_{2n} \leq \ldots \leq U_{nn}$ is the ordered sample of $n$ independent
Uniform$(0,1)$ random variables.
Introducing the function $\Psi(x) = 2^{\{ \log_2 x \}}$, i.e.~it grows linearly from 1 to 2 on each
interval $[2^j, 2^{j+1})$, $j = 1,2, \ldots$, we have
\[
\left( X_{1n}, \ldots, X_{nn} \right)
\stackrel{\mathcal{D}}{=}
\left( \frac{\Psi(U_{1n})}{U_{1n}}, \ldots,  \frac{\Psi(U_{nn})}{U_{nn}} \right).
\]

In the following we frequently use the simple facts $\Psi(u) / u = 2^{- \lfloor \log_2 u \rfloor}$ and
$2^{-\lfloor \log_2 u \rfloor} > x$ if and only if $u < 2^{-\lfloor \log_2 x \rfloor}$.
The density function of $U_{r+1,n}$ is $\binom{n}{r+1} (r+1) u^r (1-u)^{n-r-1}$, therefore
\begin{equation} \label{eq:I3}
\p \left\{ \frac{\Psi(U_{r+1,n})}{U_{r+1,n}} > x \right\} = \p \{ U_{r+1,n} < 2^{-\lfloor \log_2 x \rfloor} \}
\sim \binom{n}{r+1}  \frac{2^{\{ \log_2 x \} (r+1) }}{x^{r+1}}.
\end{equation}
Considering the asymptotics, write
\[
\begin{split}
\p \{ S_{n,r} > x \} 
& = \sum_{m=1}^\infty \p \left\{ S_{n,r+1} > x - 2^m, \, \frac{\Psi(U_{r+1,n})}{U_{r+1,n}} = 2^m  \right\}  \\
& = \sum_{m=1}^{\lfloor \log_2 x \rfloor -1}
\p \left\{ S_{n,r+1} > x - 2^m, \, \frac{\Psi(U_{r+1,n})}{U_{r+1,n}} = 2^m  \right\} \\
& \phantom{=} \, + \p \left\{ S_{n,r+1} > x - 2^{\lfloor \log_2 x \rfloor},
\, \frac{\Psi(U_{r+1,n})}{U_{r+1,n}} = 2^{\lfloor \log_2 x \rfloor}  \right\} \\
& \phantom{=} \, + \p \left\{ \frac{\Psi(U_{r+1,n})}{U_{r+1,n}} > x \right\} \\
& =: I_1 + I_2 + I_3.
\end{split}
\]
For $m \leq \lfloor \log_2 x \rfloor -1$ we have $x - 2^m \geq x / 2$, thus, by (\ref{eq:I3}) the first sum
\[
I_1 \leq \lfloor \log_2 x \rfloor \p \{ S_{n,r+1} > x/2  \} = O(x^{-(r+2)} \ln x ).
\]
For $I_2$ we have
\begin{equation} \label{eq:I2}
\begin{split}
I_2 & = \p \left\{  \sum_{i=r+2}^n \frac{\Psi(U_{in})}{U_{in}} > x (1 - 2^{-\{ \log_2 x \}} ) 
\Big|  \frac{\Psi(U_{r+1,n})}{U_{r+1,n}} = 2^{\lfloor \log_2 x \rfloor}  \right\} \\
& \phantom{=} \, \times \p \left\{ \frac{\Psi(U_{r+1,n})}{U_{r+1,n}} = 2^{\lfloor \log_2 x \rfloor}  \right\} \\
& \sim \p \left\{ S_{n-r-1} >  x (1 - 2^{-\{ \log_2 x \}}) \right\} 
\binom{n}{r+1} (2^{r+1} - 1) \frac{2^{(r+1) \{ \log_2 x \}}}{x^{r+1}}.
\end{split}
\end{equation}
Here we used the simple fact that conditioning on
$U_{n,r+1} \to 0$
\[
(U_{n,r+2}, \ldots, U_{nn}) \stackrel{\mathcal{D}}{\longrightarrow}
(U_{1,n-r-1}, \ldots, U_{n-r-1, n-r-1} ).
\]
Combining (\ref{eq:I2}) and (\ref{eq:I3}) formula (\ref{eq:exacttail}) follows. To show (\ref{eq:tail-as})
notice that if $\{ \log_2 x \} > \delta$ and $x \to \infty$, then 
$x (1 - 2^{-\{ \log_2 x \}}) \to \infty$, which implies the convergence
$\p \left\{ S_{n-r-1} >  x (1 - 2^{-\{ \log_2 x \}}) \right\} \to 0$.
\end{proof}

When $r = 0$ the result describes the tail behavior of the untrimmed sum $S_n$.
In Figure \ref{fig:df} the oscillatory behavior of $\p \{ S_n > x \}$
is clearly visible. We also see that at each power of 2 there is a large jump, that
is where the asymptotic (\ref{eq:tail-as}) fails.
\begin{figure}[!ht]
\begin{center}
\includegraphics[width=10cm]{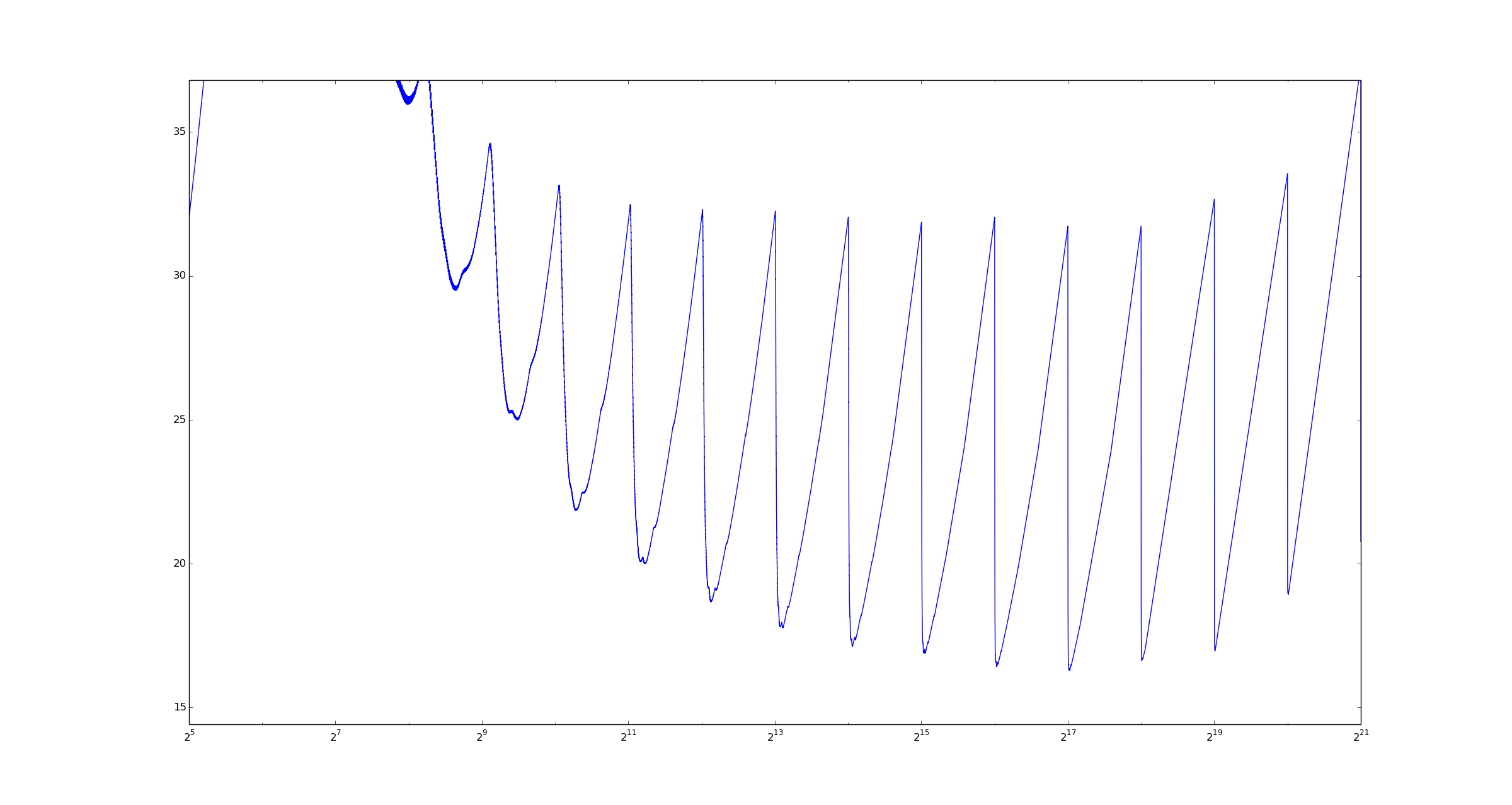}
\caption{The function $x \cdot \p \{ S_{16} > x \}$ in a
logarithmic scale.} \label{fig:df}
\end{center}
\end{figure}

We mention some important consequences.

In the untrimmed case Theorem \ref{th:Snr-tail} gives
\[
\p \{ S_{n} > x \} \sim \frac{2^{\{\log_2 x \}}}{x} n 
\left( 1 + \p \left\{ S_{n-1} >  x (1 - 2^{-\{ \log_2 x \}}) \right\}  \right),
\]
which readily implies that for any $n \geq 1$ we have
\[ 
n  = \liminf_{x \to \infty} x  \p \left\{ S_n >  x \right\} <
\limsup_{x \to \infty} x  \p \left\{ S_n > x \right\} = 2 n.
\] 

Since $x \p \{ X > x \} = 2^{\{ \log_2 x \} }$, $x \geq 2$, we have
\[
\lim_{x \to \infty, \{ \log_2 x \} \geq \delta}
\frac{ \p \{ S_n > x \} } {\p \{ X > x \} } = n.
\]
This convergence also shows that (\ref{eq:tail-as}) does not hold without
the restriction, since by (\ref{eq:subexp-prop2}) that would imply the
subexponentiality of $F $.

For $c> 1$ fixed as $m \to \infty$
\[
1 - 2^{ -\{ \log_2 ( 2^m + c  ) \}  }
\sim  1 - e^{- c 2^{-m} }
\sim  c 2^{-m}.
\]
Therefore from (\ref{eq:exacttail}) we obtain that for any $c > 1$
\begin{equation} \label{eq:finer-as}
\p \{ S_{n,r} > 2^m + c \} \sim 2^{-m (r+1)} \binom{n}{r+1}
\left(  1 + \left( 2^{r+1} - 1 \right) \p \{ S_{n-r-1} > c \} \right),
\end{equation}
as $m \to \infty$. For the maximum term we have
$\p \{ X_n^* > x \} \sim n \p \{ X > x \}$, and so
(\ref{eq:finer-as}) with $r=0$ gives
\[
\lim_{m \to \infty}
\frac{\p \left\{ S_n >  2^m + c  \right\}}{\p \{ X_n^* > 2^m + c \} }
= 1 +  \p \left\{ S_{n-1} >  c   \right\}.
\]
If $c = c(m)$ tends to infinity arbitrarily slowly, then the limit above
is 1, that is the St.~Petersburg distribution is very close to having the
subexponential property.

\section{Properties of the limit} \label{sect:ntoinf}

\subsection{Properties of the 1-trimmed limit}

In the following we determine the possible limit distributions of the 
1-trimmed sum, and we investigate the limit.

First we introduce some notation. Given that $X \leq 2^{k}$ for $i \leq k$ we have
$\p\left\{ X = 2^{i} | X \leq 2^{k} \right\} =
2^{-i} / (1 - 2^{-k})$. Introduce the
corresponding distribution function
\[
\begin{split}
F_k  (x) & = \p\left\{ X \leq x | X \leq 2^{k} \right\} 
=
\begin{cases}
\frac{1}{1- 2^{-k}} \left[ 1 - \frac{ 2^{\{ \log_2 x\}}}{x} \right],
& \text{for } x \in [2, 2^{k}], \\
1, & \text{for } x \geq 2^{k}.
\end{cases}
\end{split}
\]
In the following $X^{(k)}, X_1^{(k)}, X_2^{(k)}, \ldots,$ are i.i.d.~random variables with
distribution function $F_k $, and $S_n^{(k)}$ stands for their partial sums.
For the moments we have (see (29) in \cite{FuGyKe14})
\begin{equation} \label{eq:trunc-moment}
\begin{split}
\E (X^{(k)})^{\ell} & =
\frac{1}{1 - 2^{-k}} \sum_{i=1}^k 2^{i \ell} 2^{-i}
=
\begin{cases}
\frac{2^{\ell-1}}{1- 2^{-k}}
 \frac{2^{\left( \ell - 1 \right) k} - 1}{2^{\ell-1} - 1},&
\text{for } \ell \geq 2, \\
\frac{k }{1- 2^{-k}} , & \text{for } \ell = 1.
\end{cases}
\end{split}
\end{equation}

Introduce the infinitely divisible random variables
$W_{j,\gamma}$, $j \in \Z, \gamma \in [1/2,1]$, with characteristic function
\begin{equation} \label{eq:wjg-def}
\varphi_{j,\gamma} (t) = \E e^{\i t W_{j, \gamma} } =
\exp \bigg[ \i t u_{j, \gamma}  +
\int_0^\infty \left( e^{\i t x} - 1 - \i t x \right) \d L_{j, \gamma} (x) \bigg],
\end{equation}
with
\[ 
L _{j, \gamma}(x) =
\begin{cases}
\gamma 2^{-j} - \frac{2^{ \{ \log_2 (\gamma x) \} }}{x}, & \text{for }
x < 2^{j} \gamma^{-1}, \\
0, & \text{for } x \geq 2^{j} \gamma^{-1},
\end{cases}
\] 
and $u_{j,\gamma}  = j - \log_2 \gamma$.
According to Corollary 2 in \cite{FuGyKe14}
$W_{j,\gamma}$'s are the possible subsequential limits of
$S_n^{(\lceil \log_2 n \rceil + j )}$, more precisely
\[ 
\sup_{x \in \R}\left| \p \left\{ \frac{S_n^{(\lceil \log_2 n \rceil + j )}}{n }
- \log_2 n \leq x  \right\}
-  G_{j, \gamma_n}  (x)  \right| \to 0.
\] 
We show an exponential tail bound for the sums conditioned on the maxima.

\begin{lemma} \label{lemma:sntail}
For $x \ge 0$, put
\[
h(x)=(2+x)\ln \left(1+ \frac{x}{2} \right)- x.
\]
For any $n \geq 1$,  $j \geq 1 - \lceil \log_2 n \rceil$ and $x\ge 0$, we have that
\[
\p\left\{ S_n^{( \lceil \log_2 n \rceil + j )} -
\E S_n^{( \lceil \log_2 n \rceil + j )} > n x \right\}
\leq e^{-\frac{h(x)}{\eta_{j, \gamma_n}}},
\]
where 
\begin{equation} \label{eq:eta-def}
\eta_{j,\gamma}=2^j \gamma^{-1}.
\end{equation}
\end{lemma}

\begin{proof}
For any $\lambda>0$, we apply the Chernoff bounding
technique:
\[
\begin{split}
& \p\left\{
n^{-1} \left( S_{n}^{(\lceil \log_2 n \rceil + j)} - \E S_{n}^{(\lceil \log_2 n \rceil + j)} \right)
> x  \right\} \\
&=\p\left\{
n^{-1}  S_{n}^{(\lceil \log_2 n \rceil + j)}
> x  +\E X^{(\lceil \log_2 n \rceil + j)}\right\}\\
&\leq  {e^{-\lambda \left(x+\E X^{(\lceil \log_2 n \rceil + j)}\right)}} \E \exp \left[ \lambda n^{-1}
S_{n}^{(\lceil \log_2 n \rceil + j)}
\right]   \\
& = e^{-\lambda \left(x+\E X^{(\lceil \log_2 n \rceil + j)}\right)} \left( \E  \exp \left[ \frac{\lambda}{n}
 X^{(\lceil \log_2 n \rceil + j)} \right]
\right)^{n}.
\end{split}
\]
One has that
\begin{align*}
\E  \exp \left[ \frac{\lambda}{n}
 X^{(\lceil \log_2 n \rceil + j)}   \right]
& = 1+\frac{\lambda}{n}\E X^{(\lceil \log_2 n \rceil + j)} + \sum_{\ell=2}^{\infty}
\frac{ \lambda^{\ell}\E \left\{(X^{({\lceil \log_2 n \rceil + j})})^{\ell} \right\}}{n^\ell \, \ell !}\\
& \leq  1+ \frac{\lambda}{n}\E X^{(\lceil \log_2 n \rceil + j)} + 2 \sum_{\ell=2}^{\infty}
\frac{ \lambda^{\ell}\left( \frac{2^{j}}{\gamma_n}\right)^{\ell-1} }{n\ell !}\\
& =  1+ \frac{\lambda}{n}\E X^{(\lceil \log_2 n \rceil + j)} + \frac 2n
\frac{ e^{\lambda\eta_{j, \gamma_n}}-1-\lambda\eta_{j, \gamma_n}}{\eta_{j, \gamma_n}}\\
& \le \exp \left[ \frac{\lambda}{n}\E X^{(\lceil \log_2 n \rceil + j)} + \frac 2n
\frac{ e^{\lambda\eta_{j, \gamma_n}}-1-\lambda\eta_{j, \gamma_n}}{\eta_{j, \gamma_n}}\right],
\end{align*}
where we used that by (\ref{eq:trunc-moment})
\[
\E  \left( X^{( \lceil \log_2 n \rceil + j )} \right)^\ell  =
\frac{1}{1- 2^{-(\lceil \log_2 n \rceil + j)}}
\frac{2^{\ell - 1}}{2^{\ell -1 } - 1}
\left[ \left( \frac{n 2^j}{\gamma_n} \right)^{\hspace{-3pt}  \ell -1} \hspace{-5pt}  - 1 \right]
\le 2\left( \frac{n 2^j}{\gamma_n} \right)^{\hspace{-3pt} \ell -1} \hspace{-4pt},
\]
$\ell \geq 2$. Therefore
\begin{align*}
\p\left\{
\frac{S_{n}^{(\lceil \log_2 n \rceil + j)} - \E S_{n}^{(\lceil \log_2 n \rceil + j)}}{n}
>  x  \right\}
& \le   \exp \left[ 2
\frac{ e^{\lambda\eta_{j, \gamma_n}}-1-\lambda\eta_{j, \gamma_n}}{\eta_{j, \gamma_n}} - \lambda x \right].
\end{align*}
With the choice
$\lambda= \left[ \ln \left( 1 + \frac{x}{2} \right) \right] / \eta_{j, \gamma_n}$
the lemma is proved.
\end{proof}

\begin{remark}
Note that
$h(x) \sim x \ln x$, as $x \to \infty$, therefore the upper bound for large
$x$ is approximately $\exp \left[-\gamma 2^{-j} x \ln x  \right]$.

Applying the elementary inequalities
\[
\frac{u}{1+u/2}\le \ln (1+u) \le u, \quad u \geq 0,
\]
one has that
\[
\frac{x^2}{4+x}\le h(x) \le \frac{x^2}{2},
\]
and so for any $x \geq 0$
\[
e^{-\frac{x^2}{2\eta_{j, \gamma_n}}} \le
e^{-\frac{h(x)}{\eta_{j, \gamma_n}}} \leq
e^{-\frac{x^2}{(4+x)\eta_{j, \gamma_n}}}.
\]
Since $h(x) = x^2/4 + o(x^2)$ as $x \to 0$, for small $x\ge 0$,
we have 
$e^{-\frac{h(x)}{\eta_{j, \gamma_n}}}
\approx e^{-\frac{x^2}{4\eta_{j, \gamma_n}}}$.
\end{remark}

\begin{remark}
We note that this exponential inequality (and its straightforward extension to
generalized St.~Petersburg games) allows us to show that
arbitrary powers of the random variables
$(S_n^{(k_n)} - \E S_n^{(k_n)})/ \Var S_n^{(k_n)}$ are uniformly integrable, whenever
$\log_2 n - k_n \to \infty$. The latter implies that in Propositions 2 and 3 in \cite{FuGyKe14}
not only distributional convergence, but also moment convergence holds.
\end{remark}

Going back to (\ref{eq:wjg-def}) note that each $W_{j,\gamma}$ has finite exponential moment of any order.
We pointed out in \cite{FuGyKe14} that the distribution function
$G_{j,\gamma}(x) = \p \{ W_{j, \gamma} \leq x \}$ is infinitely
many times differentiable.
Since the support of the  L\'evy measure is bounded, according to 
Theorem 26.1 in \cite{Sato}
for the tail behavior of $W_{j,\gamma}$ we have the following.
For any $0 < c < \gamma/2^j$ 
\[
\E \exp \left\{ c W_{j,\gamma} |\ln W_{j,\gamma}| \right\} < \infty,
\]
and so
\[
\p \left\{ |W_{j,\gamma} | > x \right\} = o( \exp\{ - c x \ln x \} ),
\quad \text{as } \, x \to \infty,
\]
while for $c > \gamma/2^j$
\[
\E \exp \left\{ c W_{j,\gamma} |\ln W_{j,\gamma}| \right\} = \infty,
\]
and
\[
\p \left\{ |W_{j,\gamma} | > x \right\} \exp\{ c x \ln x \}  \to \infty,
\quad \text{as } \, x \to \infty.
\]
This result combining with Proposition 5 in \cite{FuGyKe14} implies that
the tail bound in Lemma \ref{lemma:sntail} is optimal.

Expanding the exponential in Taylor-series and changing the
order of the summation we obtain
\begin{equation} \label{eq:f}
\log \varphi_{j,\gamma}(t) = \i t \log_2 \eta_{j, \gamma} +
\sum_{k=2}^\infty \frac{( \i t )^k}{k!} \eta_{j, \gamma}^{k-1}
\frac{2^{k-1}}{2^{k-1} - 1}=: \i t \log_2 \eta_{j, \gamma} +
f_{\eta_{j,\gamma}}(t),
\end{equation}
with $\eta_{j,\gamma} = 2^j / \gamma$ as in (\ref{eq:eta-def}).
The distribution of $W_{j,\gamma}$
depends only on the single parameter $\eta_{j,\gamma} = 2^j/\gamma$. Denote
$Z_{\eta}$ a random variable with the characteristic function $e^{f_\eta (t)}$. Then,
by the definition of $f_\eta$
\[
\E e^{\i t \frac{Z_\eta}{\eta}} = e^{f_\eta(t/\eta)} = e^{f_1(t)/\eta},
\]
thus from the properties of $Z_1$ we can derive the properties of $Z_\eta$, for any $\eta$.
For example, for the density function $g_\eta$ of $Z_\eta$ we have
\[
\begin{split}
g_\eta(x)& =\frac{1}{2\pi}\int_{-\infty}^{\infty}e^{f_{\eta} (t) }e^{-\i tx} \d t
=\frac{1}{2\pi}\int_{-\infty}^{\infty}e^{f_{1} (t\eta)/\eta }e^{-\i tx} \d t \\
& =\frac{1}{2\pi}\frac{1}{\eta}\int_{-\infty}^{\infty}e^{(f_{1} (t)-\i tx)/\eta} \d t.
\end{split}
\]
Thus, $g_\eta$ can be derived from the characteristic function $e^{f_{1}}$ of $Z_{1}$
by a simple transformation. It also shows, that the proper scaling is the variance
instead of the standard deviation.

From (\ref{eq:f}) it is apparent that
\[
\frac{W_{j,\gamma} -  \log_2 \eta_{j,\gamma}}{ \sqrt{2 \eta_{j,\gamma}}}
\stackrel{\mathcal{D}}{\longrightarrow} \mathrm{N}(0,1),
\quad \text{as } \eta_{j,\gamma} \to 0,
\]
while
\[
\frac{W_{j,\gamma}} {  \eta_{j,\gamma}}
\stackrel{\p}{\longrightarrow} 0,
\quad \text{as } \eta_{j,\gamma} \to \infty.
\]
These limit theorems are in complete accordance with
Proposition 3 in \cite{FuGyKe14}, which states that conditioning on small
maximum the limit is normal, and with Proposition 7 \cite{FuGyKe14},
which states that conditioning on large maximum the limit is deterministic.

By Proposition 6 in \cite{FuGyKe14} for each $j \in \Z$
\begin{equation} \label{eq:typical-max-cond}
\sup_{x \in \R}
\left| \p \left\{ \frac{S_n}{n } - \log_2 n \leq x
\Big| X_n^* = 2^{\lceil \log_2 n \rceil + j } \right\}
- \widetilde G_{j, \gamma_n}  (x)  \right| \to 0,
\end{equation}
where
\[ 
\widetilde G_{j, \gamma}  (x) =
\sum_{m=1}^\infty G_{j-1,\gamma}  \left( x - m 2^{j}/\gamma \right) r_{j, \gamma} (m),
\] 
with
\[
r_{j,\gamma}(m)  = \frac{(2^{-j} \gamma )^m}{m!} \left( e^{2^{-j} \gamma} - 1 \right)^{-1},
\quad m \geq 1.
\] 
The distribution $(r_{j,\gamma}(m))_{m \geq 1}$ is  
Poisson-distribution conditioned on being nonzero. 
From Proposition II.2.7 in \cite{SvH} it follows that 
this distribution is \emph{not} infinitely divisible.
In Theorem 1 in \cite{FuGyKe14} we showed that
for any $\gamma \in [1/2, 1]$
\[
G_{\gamma} (x) =
\sum_{j=-\infty}^\infty \widetilde G_{j, \gamma}  (x) p_{j,\gamma}.
\]
For the trimmed sum we have the following the merging theorem,
together with the infinite series representation of the limiting distribution
function.

\begin{theorem} \label{th:trimmed}
We have
\[
\sup_{x \in \mathbb{R}} \left|
\p\left\{ \frac{S_n-X_n^*}{n} - \log_2 n \leq x \right\}  -
G^*_{\gamma_n}(x) \right| \to 0, \quad \text{as } n \to \infty,
\]
where
\begin{equation} \label{eq:G-star-df}
G^*_{\gamma} (x) =
\sum_{j=-\infty}^\infty
\sum_{m=1}^\infty G_{j-1,\gamma} ( x - (m-1) 2^{j}/\gamma ) \, r_{j, \gamma} (m) \,
p_{j,\gamma}, \quad \gamma \in (1/2,1].
\end{equation}
\end{theorem}

\begin{proof}
Since (\ref{eq:typical-max-cond}) holds uniformly in $x$, we obtain
\[
\sup_{x \in \R}
\left| \p \left\{ \frac{S_n - X_n^*}{n } - \log_2 n \leq x \Big| X_n^* = 2^{\lceil \log_2 n \rceil + j } \right\}
- \widetilde G_{j, \gamma_n}  (x + 2^j / \gamma_n )  \right| \to 0.
\]
Using (\ref{eq:max-pr-merge}) and the same conditioning as in the proof of
Theorem 1 in \cite{FuGyKe14} we obtain the statement.
\end{proof}

This result implies that, as usual in this setup, along subsequences
there is distributional convergence. For the subsequence $n_k = \lfloor \gamma 2^k \rfloor$,
$\gamma \in [1/2, 1]$, which in fact covers all the possible limits, this
was shown by Gut and Martin-L\"of in Theorem 6.1 \cite{GML}.

The infinite series representation of $G_\gamma$ in Theorem 1 in \cite{FuGyKe14}
is in fact equivalent to the distributional representation
\[
W_\gamma \stackrel{\mathcal{D}}{=} W_{Y_\gamma -1, \gamma} + M_{Y_\gamma, \gamma}
2^{Y_\gamma} \gamma^{-1},
\]
where $(W_{j,\gamma})_{j \in \Z}$, $(M_{j,\gamma})_{j \in \Z}$ and $Y_\gamma$
are independent random variables, $Y_\gamma$ has probability distribution
$(p_{j,\gamma})_{j \in \Z}$, $M_{j,\gamma}$ has Poisson($\gamma 2^{-j}$) distribution,
conditioned on not being 0, and $W_{j,\gamma}$ is an infinitely divisible distribution
given in (\ref{eq:wjg-def}). 
Let $W_\gamma^*$ be a random variable with distribution function $G_\gamma^*$. Then,
the same way (\ref{eq:G-star-df}) reads as
\[ 
W_\gamma^* \stackrel{\mathcal{D}}{=} W_{Y_\gamma -1, \gamma} + (M_{Y_\gamma, \gamma} -1)
2^{Y_\gamma} \gamma^{-1}.
\] 
Looking at the infinitely divisible random variable $W_\gamma$ as a semistable L\'evy
process at time 1, the meaning of the representation above is the following.
The value $2^{Y_\gamma} / \gamma$ corresponds to the maximum jump, $M_{Y_\gamma, \gamma}$
is the number of the maximum jumps, and $W_{Y_\gamma -1, \gamma}$ has the law of the
L\'evy process conditioned on that the maximum jump is strictly less than $2^{Y_\gamma} / \gamma$.
This kind of distributional representations for general L\'evy processes were obtained
by Buchmann, Fan and Maller, see Theorem 2.1 in \cite{BFM}.

\subsection{Representation of the $r$-trimmed limit}

Let $E_k, k=1, 2, \ldots$ be i.i.d.~Exp(1) random variables and $Z_k=E_1 + \ldots + E_k$.

\begin{lemma} \label{lemma:Ymoment}
For any $\gamma>0$, the sum
\begin{equation} \label{eq:defYr}
Y_{r,\gamma} = \sum_{k=r+1}^\infty \left(\frac{\Psi(Z_k/\gamma)}{Z_k} - \frac{\Psi (k/\gamma)}{k}\right)
\end{equation}
converges absolutely with probability 1 and its sum belongs to $L_p$ for any $1\le p < r+1$.
\end{lemma}

\begin{proof}
We have
\begin{align} \label{2}
& \left|\frac{\Psi(Z_k/\gamma)}{Z_k} - \frac{\Psi (k/\gamma)}{k}\right|\le \Psi (Z_k/\gamma) \left| \frac{1}{Z_k} 
- \frac{1}{k}\right|+ |\Psi(Z_k/\gamma)-\Psi(k/\gamma)| \frac{1}{k} \nonumber \\
& \le 2 |Z_k-k| \frac{1}{kZ_k} + |\Psi(Z_k/\gamma)-\Psi(k/\gamma)| \frac{1}{k} =: I_k+J_k. 
\end{align}
By the H\"older inequality we have for any $p \geq 1$  and any $P, Q>1$ with $1/P+1/Q=1$,
\begin{equation}\label{3}
\E (|Z_k-k|^p (kZ_k)^{-p}) \le k^{-p} \left( \E (|Z_k-k|^{pP}\right)^{1/P} (\E Z_k^{-pQ})^{1/Q}.
\end{equation}
By the Rosenthal inequality (\cite[Theorem 2.9]{Petrov}) we have 
\begin{equation}\label{4}
\E (|Z_k-k|^{pP}) \leq c_1 k^{pP/2},
\end{equation}
with some constant $c_1 > 0$ depending only on $pP$. In the following $c_2, c_3, \ldots$ are
positive constants, whose values are not important.
On the other hand, $Z_k$ is $\mathrm{Gamma}(k, 1)$ distributed and thus for any $\beta>0$ we have
\[
\E (Z_k^{-\beta})= \int_0^\infty \frac{1}{x^{\beta}} \frac{ x^{k-1}}{\Gamma (k)} e^{-x} \d x
= \frac{\Gamma (k-\beta)}{\Gamma (k)} \leq c_2 k^{-\beta},
\]
for $k>\beta$, and thus in (\ref{3}) we have
\begin{equation}\label{5}
\E (Z_k^{-pQ})^{1/Q} \leq c_3 k^{-p}, \qquad \text{for} \ k>pQ.
\end{equation}
Since $p<r+1$, we can choose $Q>1$ so close to 1 that for $k\ge r+1$ we have $k>pQ$ and thus (\ref{5}) holds.
Choosing $Q$ close to 1 will make $P=Q/(Q-1)$ very large, but (\ref{4}) is still valid. Therefore, the left hand
side of (\ref{3}) is $\leq c_4 k^{-3p/2}$, and consequently in (\ref{2}) we have
\begin{equation}\label{6}
\| I_k \|_p \leq  c_5 k^{-3/2},  \quad \text{for} \ k \ge r+1.
\end{equation}

To estimate $J_k$ we first observe that by large deviation theory we have $|Z_k-k|\le k^{2/3}$
except on a set $A_k$ with $\p \{ A_k \} \le a \exp (-k^\delta)$ $(k\ge 1)$ for some absolute constants
$a>0$, $\delta>0$. 
To estimate the difference $\Psi(Z_k/\gamma) - \Psi(k/\gamma)$ we have to make sure that
$Z_k/\gamma$ and $k/\gamma$ fall into the same dyadic interval.
Note that when $k/\gamma$ or $Z_k/\gamma$ is close to a discontinuity point of $\Psi$, i.e.~to
an integer power of 2, then we cannot give a good estimate.
Therefore assume that $2^j + 2^{2(j+1)/3} \leq k/\gamma \leq 2^{j+1} - 2^{2(j+1)/3}$.
Then on the set $A_k^c$ we have $Z_k/\gamma \in [2^j, 2^{j+1}]$
and thus  $|\Psi (Z_k/\gamma) - \Psi(k/\gamma)| \leq 2^{1-j} |Z_k - k| \leq 4 k^{-1/3}$.
Therefore, for such $k$'s in (\ref{2}) we have
$J_k\leq 4 k^{-4/3}$ except on $A_k$,  and on $A_k$ trivially $J_k \le 2$. Thus we proved
\[ 
\| J_k \|_p \leq c_6 k^{-4/3},  \qquad \text{for} \ k \ge r+1, k \in M,
\] 
with
\[
M = \cup_{j=1}^\infty \left[ \gamma (2^j + 2^{2(j+1)/3}), \gamma (2^{j+1} - 2^{2(j+1)/3}) \right].
\]
For $k \not\in M$ we only have that $J_k \leq 2/k$, but since there are not so many such $k$'s it
is enough, more precisely
\[
\sum_{k \not \in M} \|J_k \|_p \leq \sum_{j=1}^\infty 4 \cdot 2^{2(j+1)/3} 2^{-j} < \infty. 
\]
Consequently by (\ref{2}), (\ref{6}) we proved that
\[
\sum_{k=r+1}^\infty \left\|\frac{\Psi(Z_k/\gamma)}{Z_k} - \frac{\Psi (k/\gamma)}{k}\right\|_p <\infty, \qquad \text{for} \ p<r+1,
\]
completing the proof of the lemma.
\end{proof}

In the following theorem we need that the convergence of $\gamma_{n_k}$ to some limit $\gamma$ is fast enough.
However, the natural subsequence $n_k = \lfloor \gamma 2^k \rfloor$ satisfies this condition.

\begin{theorem} \label{th:trimlim}
Assume that $\gamma_{n_k} = \gamma + O(n_k^{-1/5})$ along the subsequence $n_k$, where $\gamma \in (1/2, 1]$.
Then for any $r \geq 0$
\[ 
\frac{1}{n_k}S_{n_k,r}- a_{n_k,\gamma}^{(r)}
\stackrel{\mathcal{D}}{\longrightarrow} Y_{r,\gamma},
\] 
with centering sequence
\[ 
a_{n,\gamma}^{(r)} = \sum_{j=r+1}^n \frac{\Psi(j/\gamma)}{j}.
\] 
\end{theorem}

\begin{proof}
We rewrite the representation (\ref{eq:quantile-uni}) in terms of the Poisson process determined by $(E_i)_{i \in \N}$.
Since for $n$ fix
\[
(U_{1n}, U_{2n}, \ldots, U_{nn} ) \stackrel{\mathcal{D}}{=}
\left( \frac{Z_1}{Z_{n+1}}, \frac{Z_2}{Z_{n+1}}, \ldots, \frac{Z_n}{Z_{n+1}} \right),
\]
we obtain
\[
\begin{split}
\left( X_{1n}, \ldots, X_{nn} \right)
& \stackrel{\mathcal{D}}{=}
\left( \frac{Z_{n+1}}{Z_1} \Psi(Z_1/Z_{n+1}), \ldots,  \frac{Z_{n+1}}{Z_n} \Psi(Z_n/Z_{n+1}) \right) \\
& =: \left( X_{1n}^*, \ldots, X_{nn}^* \right).
\end{split}
\]
By the strong law of large numbers $Z_{n+1}/n \to 1$ a.s.\ whence it follows
\begin{equation}\label{leftside}
X_{1,n}^* = \frac{n}{Z_1} \Psi\left(\frac{Z_1}{n}\right) (1+o(1))
\qquad \text{a.s.}
\end{equation}
Now if along a subsequence $\gamma_{n_k} \to \gamma \in (1/2,1]$
we obtain, using (\ref{leftside}),
\[
\frac{X_{1,n_k}^*}{n_k} \to 
\frac{1}{Z_1} \Psi\left(\frac{Z_1}{\gamma}\right)
\qquad \text{a.s.}
\]
Note that although $\Psi$ is not continuous, the probability that $Z_1/\gamma$ falls in $2^\Z$ is zero.
Similar formulas apply for $X_{j, n_k}^*/n_k$, and thus we get for any fixed $K \ge 1$
\[ 
\frac{1}{n_k} (X_{1,n_k}, \ldots, X_{K,n_k})
\overset{\mathcal{D}}{\longrightarrow}
\left( \frac{\Psi\left(Z_1/\gamma\right)}{Z_1}, \ldots,
\frac{\Psi\left(Z_K/ \gamma\right)}{Z_K}\right).
\] 
Observe that 
\begin{equation}\label{a1}
\frac{1}{n} S_{n,r} \stackrel{\mathcal{D}}{=}
\sum_{j=r+1}^n \frac{\Psi(Z_j/Z_{n+1})}{nZ_j/Z_{n+1}}=
\frac{Z_{n+1}}{n} \sum_{j=r+1}^n \frac{\Psi(Z_j/Z_{n+1})}{Z_j}.
\end{equation} 
Now by (\ref{a1})
\[ 
\begin{split}
\frac{1}{n_k} S_{n_k,r} - a_{n_k,\gamma}^{(r)} & \overset{\mathcal{D}}{=}
\left(\frac{Z_{n_k+1}}{n_k}-1\right)
\sum_{j=r+1}^{n_k} \frac{\Psi(Z_j/Z_{n_k+1})}{Z_j} \\
& \phantom{=} \, +
\sum_{j=r+1}^{n_k} \left(\frac{\Psi(Z_j/Z_{n_k+1})}{Z_j}-\frac{\Psi(j/\gamma)}{j}\right) \\
& =: U_{n_k}+V_{n_k}.
\end{split}
\] 
By the strong law of large numbers, the sum in $U_n$ is $O(\ln n)$ a.s., further Chebyshev's inequality
implies $|Z_{n+1}/n -1|=O_P(n^{-1/2})$, and thus $U_n\to 0$ in probability.
Rewrite $V_{n_k}$ as
\begin{equation} \label{eq:Vnk}
V_{n_k} = 
\sum_{j=r+1}^{n_k} \hspace{-3pt} \left(\frac{\Psi(Z_j/Z_{n_k+1}) - \Psi(Z_j/\gamma)}{Z_j} \right)
+ \sum_{j=r+1}^{n_k} \hspace{-3pt} \left( \frac{\Psi(Z_j/\gamma)}{Z_j} - \frac{\Psi(j/\gamma)}{j} \right).
\end{equation}
The second sum obviously converges almost surely to $Y_{r,\gamma}$. 

We show that the first sum converges
a.s.~to 0. Using $\E |Z_k-k|^4=O(k^2)$, the Markov inequality and the Borel--Cantelli lemma,
it follows that $Z_{n+1}=n+O(n^{4/5})$ a.s. Here, and in the sequel, constants in the $O$
depend only on $\omega$. Thus for any $1\le j\le n$, by the logarithmic periodicity
of $\Psi$
\begin{equation}\label{a2}
\Psi(Z_j/Z_{n+1})=\Psi \big( (1+O(n^{-1/5}))Z_j/n \big)=\Psi \big( (1+O(n^{-1/5}))Z_j/\gamma_n \big).
\end{equation}
To estimate $|\Psi(Z_j/ Z_{n_k+1}) - \Psi(Z_j/\gamma)|$ we have to use the same idea as in the proof of
Lemma \ref{lemma:Ymoment}. 
Given $1\le j\le n_k$, let $m=m(j)$ be defined by $j/\gamma \in [2^m + 2^{2(m+1)/3}, 2^{m+1}- 2^{2(m+1)/3}]$. 
Then $Z_j\sim j$ a.s.\ implies  $Z_j/\gamma \in (2^m, 2^{m+1}]$, moreover
$(1+O(n_k^{-1/5}))Z_j/\gamma_{n_k} \in [2^m, 2^{m+1}]$, for $j\ge j_0(\omega)$. 
The slope of $\Psi(x)$ on this interval is
$2^{-m}\le 2/j$ and thus replacing $1+O(n^{-1/5})$ by 1 changes the last expression of (\ref{a2})
at most by $O(n^{-1/5}\, (Z_j/\gamma_n)\, (2/j))=O(n^{-1/5})$  i.e., for such $j$'s
\begin{equation}\label{a3}
|\Psi(Z_j/Z_{n_k+1})-\Psi(Z_j/\gamma)|=O \left( n_k^{-1/5} \right), \quad j_0(\omega) \leq j \leq n_k, \ j \in M,
\end{equation}
where $M = \cup_{m=1}^\infty [2^m + 2^{2(m+1)/3}, 2^{m+1}- 2^{2(m+1)/3}]$.
Here we used that $\gamma_{n_k} = \gamma + O(n_k^{-1/5})$.
We emphasize that (\ref{a3}) is not true for all $j \leq n$. For the sum of these terms
\[
\sum_{j_0(\omega) \leq j \leq n_k, j \in M} \frac{|\Psi(Z_j/Z_{n_k+1}) - \Psi(Z_j/\gamma)|}{Z_j}
= O( n_k^{-1/5} \ln n_k).
\]
The remaining indices can be estimated exactly as in the proof of Lemma \ref{lemma:Ymoment}, and
thus we obtain that the first sum in (\ref{eq:Vnk}) converges a.s.
On the other hand, for each $j$
\[
\frac{\Psi(Z_j/Z_{n_k+1})}{Z_j} \to \frac{\Psi(Z_j/\gamma)}{Z_j} \quad \text{a.s.}
\]
so the almost sure limit is 0. The convergence of the series (\ref{eq:defYr})
imply that $V_{n_k} \to Y_{r,\gamma}$ a.s.\ as $n_k \to\infty$, completing the proof of the theorem. 
\end{proof}

\subsection{Around the centering}

The definition of the centering sequence $a_{n,\gamma}^{(r)}$ is quite natural in view of
Lemma \ref{lemma:Ymoment}. However, its asymptotic behavior is not immediately clear, and more importantly,
it is not continuous as a function of $\gamma$. This is the reason why we cannot prove the merging
counterpart of Theorem \ref{th:trimlim}. In this section we gather some information about $a_{n,\gamma}^{(r)}$.

Using that $\Psi(z/\gamma)/ z = 2^{- \lfloor \log_2 (z/\gamma) \rfloor} / \gamma$ we obtain
\[
\begin{split}
a_{n,\gamma}^{(0)} & = \sum_{k=1}^n \frac{\Psi (k/ \gamma)}{k} 
= \sum_{m=0}^{ m_n -1} \sum_{k=\lceil \gamma 2^m \rceil}^{\lceil \gamma 2^{m+1} \rceil -1}
\frac{2^{-m}}{\gamma} + \frac{n - \lceil \gamma 2^{m_n} \rceil+1}{\gamma 2^{m_n}} \\
& = \!\! \sum_{m=0}^{ m_n -1} \frac{\lceil \gamma 2^{m+1} \rceil - \lceil \gamma 2^{m} \rceil }{\gamma 2^m}
+ \frac{n - \lceil \gamma 2^{m_n} \rceil+1}{\gamma 2^{m_n}} 
= \frac{n+1}{\gamma 2^{m_n}} - \frac{1}{\gamma} + \sum_{m=1}^{m_n} \frac{\lceil \gamma 2^{m} \rceil}{\gamma 2^m},
\end{split}
\]
with $m_n = \lfloor \log_2 (n+1)/\gamma \rfloor$.
Thus
\begin{equation} \label{eq:an-form}
a_{n,\gamma}^{(0)} - m_n = \frac{n+1}{\gamma 2^{m_n}} - \gamma^{-1} + 
\sum_{m=1}^{m_n} \frac{\lceil \gamma 2^{m} \rceil - \gamma 2^m}{\gamma 2^m}
\leq \frac{n+1}{\gamma 2^{m_n}} < 4.
\end{equation}
For the subsequence $n_k = 2^k$ and for $\gamma = 1$ we see that
\[
a_{n_k,1}^{(0)} - \log_2 n_k = \frac{1}{n_k} \to 0,
\]
and comparing (\ref{eq:sum-merge}) and Theorem \ref{th:trimlim}, this means that $W_1 = Y_{0,1}$. 

In general, for any $\gamma \in (1/2, 1]$ it is easy to see that $m_{n_k} - k \to 0$ 
on the subsequence $n_k = \lfloor \gamma 2^k \rfloor$, that is
\[
m_{n_k} - \log_2 n_k \to - \log_2 \gamma.
\]
Combined with (\ref{eq:an-form}) this implies that our centering is (apart from a constant)
the same as the usual $\log_2 n$ centering in (\ref{eq:sum-merge}).
More precisely, for any $\gamma \in (1/2, 1]$, $n_k = \lfloor \gamma 2^k \rfloor$,
\begin{equation} \label{eq:anlimit}
a_{n_k, \gamma}^{(0)} - \log_2 n_k \to
1 - \gamma^{-1} + \sum_{m=1}^\infty \frac{ \lceil \gamma 2^m \rceil - \gamma 2^m}{\gamma 2^m}
- \log_2 \gamma.
\end{equation}

In the limit in (\ref{eq:anlimit}) appears the function
\[
f(\gamma) = \sum_{m=1}^\infty \frac{\lceil 2^m \gamma \rceil - 2^m \gamma}{2^m \gamma}, \quad \gamma \in (1/2,1].
\]
The function
\[
\xi(\gamma) = 2 - \frac{1}{\gamma} \sum_{k=1}^\infty \frac{k \varepsilon_k}{2^k} - \log_2 \gamma, \quad \gamma \in [1/2, 1],
\]
was introduced by Cs\"org\H{o} and Simons \cite{CsS}, see also Kern and Wedrich \cite{KW}.
Here $\gamma = \sum_{k=1}^\infty \varepsilon_k 2^{-k}$, where $\varepsilon_k$'s are the dyadic digits of $\gamma$.
This function appears naturally if one considers Steinhaus' resolution of the St.~Petersburg paradox, see \cite{CsS}.
Cs\"org\H{o} and Simons  \cite[Theorem 3.1]{CsS} show that $\xi$ is right-continuous, left-continuous
except at dyadic rationals greater than $1/2$ and has unbounded variation. Moreover, Kern and Wedrich \cite[Theorem 3.1]{KW}
proved that the Hausdorff and box-dimension of the graph of $\xi$ is 1, meaning that $\xi$ is not so irregular.
These properties remain true for our function $f$, as it turns out that
\[
f(\gamma) = \xi(\gamma) + \log_2 \gamma - 1 + \frac{1}{\gamma}.
\]
For any $\gamma = \sum_{k=1}^\infty \varepsilon_k 2^{-k}$ we consider the representation, which contains infinitely many 1's.
Then
\[
\lceil 2^m \gamma \rceil = 2^m \sum_{k=1}^m \varepsilon_k 2^{-k} + I(\exists k > m, \varepsilon_k =1),
\]
and so
\[ 
\lceil 2^m \gamma \rceil -  2^m \gamma = I(\exists k > m, \varepsilon_k =1) - \sum_{k= m +1}^\infty \varepsilon_k 2^{m-k}.
\] 
Thus
\[
\begin{split}
f(\gamma) & = \sum_{m=1}^\infty \frac{1}{2^m \gamma} \left( I(\exists k > m, \varepsilon_k =1) - 
\sum_{k=m+1}^\infty \varepsilon_k 2^{m-k} \right) \\
& = \frac{1}{\gamma} - \frac{1}{\gamma}  \sum_{m=1}^\infty \sum_{k=m+1}^\infty \varepsilon_k 2^{-k} 
= \frac{1}{\gamma} - \frac{1}{\gamma}  \sum_{k=1}^\infty (k-1) \varepsilon_k 2^{-k} \\
& = \frac{1}{\gamma} + 1  - \frac{1}{\gamma}  \sum_{k=1}^\infty k \varepsilon_k 2^{-k} 
= \xi(\gamma) + \log_2 \gamma - 1 + \frac{1}{\gamma}.
\end{split}
\]
The limit appearing in (\ref{eq:anlimit}) is exactly $\xi(\gamma)$.

\subsection{Uniform tail bound for the trimmed sums}

In this section we obtain a uniform tail bound for the centralized and normalized trimmed sum.

\begin{theorem} \label{th:trimtail}
For any $\delta \in (0,1)$,  $x \geq e$ and $n \geq 1$ there is a finite
constant $C > 0$, such that
\[
\p \left\{   \frac{S_{n,r}}{n} - a_{n,\gamma_n}^{(r)} > x \right\} 
\leq \frac{2^{r+1}}{(r+1)!} {[(1-\delta)x]^{-(r+1)}} + C \, \delta^{-(r+3/2)} x^{-(r+3/2)}.
\]
\end{theorem}

\begin{proof}
As before, consider the representation
\[
\frac{S_{n,r}}{n} - \sum_{k=r+1}^n \frac{\Psi(k/\gamma_n)}{k} = 
\sum_{k=r+1}^n \left( \frac{\Psi(U_{n,k})}{n U_{n,k}} - \frac{\Psi(k/\gamma_n)}{k} \right). 
\]
The tail of the first term is of order $x^{-(r+1)}$ (uniformly in $n$). 
We show that the $L_p$ norm of the remaining sum is bounded,
that is for $p \in (r+1, r+2)$
\begin{equation} \label{eq:snr-unibound}
\sup_n \E \left| \sum_{k=r+2}^n \left( \frac{\Psi(U_{n,k})}{n U_{n,k}} -
\frac{\Psi(k/\gamma_n)}{k} \right) \right|^p
\leq C_p < \infty.
\end{equation}
We use the same technique as in the proof of Lemma \ref{lemma:Ymoment},
the only difference is that we need a uniform bound for the 
empirical quantile process. Write
\[
\left| \frac{\Psi(U_{n,k})}{n U_{n,k}} - \frac{\Psi(k/\gamma_n)}{k} \right| \leq  
2 \left| \frac{1}{n U_{n,k}} - \frac{1}{k} \right| +
\frac{\left| \Psi(U_{n,k})  - \Psi(k/\gamma_n) \right|}{k} =: I_k + J_k.
\]
For $P, Q \geq 1$ with $1/P + 1/Q =1$
\begin{equation} \label{eq:uni-bound}
\E \left| \frac{1}{n U_{kn}} - \frac{1}{k} \right|^p \leq k^{-p}
\left( \E | nU_{kn} - k|^{pP} \right)^{1/P} \left( \E (n U_{kn})^{-pQ} \right)^{1/Q}. 
\end{equation}
For the last factor we have
\[
\E (n U_{kn})^{-\alpha} = n^{-\alpha} \frac{\Gamma(n+1) \Gamma(k-\alpha)}
{\Gamma(n+1-\alpha) \Gamma(k)} \leq c_1 \, k^{-\alpha}.
\]
For the first factor in (\ref{eq:uni-bound}) we use Mason's inequality (\cite[Proposition 2]{M}) with
$h \equiv 1$, $\nu_1 = \nu_2 =0$, and we get 
\[
\E \left| U_{k,n} - \frac{k}{n+1} \right|^\alpha \leq c_\alpha \, k^{\alpha/2} n^{-\alpha}. 
\]
Since changing $n+1$ to $n$ makes an error $n^{-2}$, we obtain
\[
\| I_k \|_p \leq c_2 \, k^{-3/2}, 
\]
which is summable. The term $J_k$ can be handled the same way as in the proof
of Lemma \ref{lemma:Ymoment}, and we obtain (\ref{eq:snr-unibound}).

Putting $w = 2^{-\lfloor \log_2 n(x+ \Psi(k/\gamma_n)/k) \rfloor}$,
more precise calculation gives
\[
\begin{split}
& \p \left\{ \frac{\Psi(U_{nk})}{n U_{nk}} - \frac{\Psi(k/\gamma_n)}{k} > x \right\}
= \binom{n}{k} k \int_0^{w}
u^{k-1} (1-u)^{n-k} \d u \\
& \leq \frac{1}{(k-1)!} \int_0^{2/x} y^{k-1} \d y \leq \frac{2^k}{k!} x^{-k}.
\end{split}
\]
Thus, using Markov's inequality combined with (\ref{eq:snr-unibound}), we obtain
for any $\delta > 0$
\[
\begin{split}
\p \bigg\{   \frac{S_{n,r}}{n} - a_{n,\gamma_n}^{(r)} > x \bigg\} 
& \leq \p \left\{ \frac{\Psi(U_{n,r+1})}{n U_{n,r+1}} 
- \frac{\Psi((r+1)/\gamma_n)}{r+1} > (1-\delta) x \right\} \\
& \phantom{=} \, 
+ \p \left\{   \frac{S_{n,r+1}}{n} - a_{n,\gamma_n}^{(r+1)} > \delta x \right\}  \\
& \leq \frac{2^{r+1}}{(r+1)!} {[(1-\delta)x]^{-(r+1)}} + C_{r+3/2} \, (x\delta)^{-(r+3/2)},
\end{split}
\]
and the statement is proved.
\end{proof}

\subsection{Tail behavior of the trimmed limit}

Introduce the notation
\[ 
A_{r,\gamma} = \sum_{k=1}^r \frac{\Psi(k/\gamma)}{k}.
\] 
Using Lemma \ref{lemma:Ymoment} we can determine the tail distribution of the trimmed limit.
The tail behavior of the semistable limit along the subsequence $2^m + c$ ($c$ fix, $m \to \infty$)
was determined by Martin-L\"of \cite[Theorem 4]{ML}. Our proof in the general $r$-trimmed setup
and also the proof of Theorem \ref{th:Snr-tail} use the same idea as Martin-L\"of: conditioning on the
maximum term.

\begin{theorem} \label{th:Gtail}
For any $r = 0,1, \ldots$
\begin{equation} \label{eq:Grtail-as}
\begin{split}
\p \{ Y_{r,\gamma} > x \} & \sim 
\frac{2^{\{ \log_2 (\gamma x) \} (r+1)}}{(r+1)! \, x^{r+1}}
\bigg[ 2^{-r-1} + (2^{r+1} - 1) \\
& \phantom{\sim} \, \times \sum_{\ell=0}^1 2^{-\ell (r+1)}
\p \left\{ Y_{0,\gamma} + A_{r,\gamma} > x\left(1 - 2^{\ell - \{ \log_2 (\gamma x) \}} \right) \right\}
\bigg].
\end{split}
\end{equation}
\end{theorem}

\begin{proof}
Simple calculation shows that for any $k \geq 1$
\[
\p \left\{ \frac{\Psi(Z_{k} / \gamma)}{Z_{k}} > x \right\}
\sim \frac{1}{k!} \frac{2^{\{ \log_2 (\gamma x) \}k}}{x^{k}},
\]
and by Lemma \ref{lemma:Ymoment}
\begin{equation} \label{eq:I1est}
\p \{ Y_{r+1, \gamma} > x \} = o(x^{-(r+3/2)}).
\end{equation}
We have for $x$ large enough
\[
\begin{split}
\p \left\{ Y_{r,\gamma} > x \right\}
& = \sum_{m = -\infty}^ \infty \p \left\{ Y_{r,\gamma} > x, \,
2^{- \lfloor \log_2 (Z_{r+1} / \gamma ) \rfloor } = 2^m \right\} \\
& = \p \left\{ Y_{r,\gamma} > x, \, 
- \lfloor \log_2 (Z_{r+1} / \gamma ) \rfloor \leq \lfloor \log_2 (\gamma x) \rfloor -1  \right\} \\
& \phantom{=} \, + \p \left\{ Y_{r,\gamma} > x, 
\, C_0 \right\} + \p \left\{ Y_{r,\gamma} > x, \, C_1  \right\} \\
& \phantom{=} \, + \p \left\{Y_{r,\gamma} > x,
- \lfloor \log_2 (Z_{r+1} / \gamma ) \rfloor \geq \lfloor \log_2 (\gamma x) \rfloor +2  \right\} \\
& =: I_1 + I_2 + I_3 + I_4,
\end{split}
\]
where we introduced the notation 
$C_\ell = \{ - \lfloor \log_2 (Z_{r+1} / \gamma ) \rfloor = \lfloor \log_2 (\gamma x) \rfloor + \ell \}$,
for $\ell=0,1$. 
Since $2^{\lfloor \log_2 (\gamma x) \rfloor - 1} \leq \gamma x / 2$, by (\ref{eq:I1est})
\[
I_1 \leq \p \left\{ Y_{r+1,\gamma} > x /2 \right\} =  o(x^{-(r+3/2)}).
\]

Conditioning on $Z_{r+1} \to 0$ we have
\[
Y_{r+1,\gamma} =  \hspace{-3pt} \sum_{k=r+2}^\infty  \hspace{-4pt}
\left( \frac{\Psi(Z_k/\gamma)}{Z_k} - \! \frac{\Psi((k-r-1)/\gamma)}{k-r-1}  \right)
+ \sum_{k=1}^{r+1} \frac{\Psi(k/\gamma)}{k}
\stackrel{\mathcal{D}}{\longrightarrow} Y_{0,\gamma} + A_{r+1,\gamma}.
\]
Therefore, for $I_2, I_3$
\[
\begin{split}
& \p \{ Y_{r,\gamma} > x, \, C_\ell \}
= \p \{ C_\ell  \} \p \left\{ \hspace{-2pt} Y_{r+1,\gamma} -  \hspace{-2pt} \frac{\Psi((r+1)/\gamma)}{r+1}> x 
(1 - 2^{\ell - \{\log_2 (\gamma x) \}})  \Big|  C_\ell  \right\} \\
& \sim  \frac{2^{ (\{ \log_2 (\gamma x) \}-\ell) (r+1)}}{(r+1)! x^{r+1}} \left( 2^{r+1} - 1 \right)
\p \left\{ Y_{0,\gamma} + A_{r,\gamma} > 
x (1 - 2^{\ell - \{\log_2 (\gamma x) \}})    \right\}.
\end{split}
\]
By passing to the limit we used the absolute continuity of $Y_{0,\gamma}$. Finally,
\[
\begin{split}
I_4 & = \p \left\{- \lfloor \log_2 (Z_{r+1} / \gamma ) \rfloor \geq \lfloor \log_2 (\gamma x) \rfloor +2  \right\} \\
& \phantom{=} \, - \p \left\{Y_{r,\gamma} \leq x, - \lfloor \log_2 (Z_{r+1} / \gamma )
\rfloor \geq \lfloor \log_2 (\gamma x) \rfloor +2  \right\}.
\end{split}
\]
Martin-L\"of (Lemma 1 in \cite{ML}) showed that the left tail is exponentially small  
(see also Theorem 5 by Watanabe and Yamamuro \cite{WY} in general),
thus for second term we have
\[
\begin{split}
& \p \left\{Y_{r,\gamma} \leq x, - \lfloor \log_2 (Z_{r+1} / \gamma ) \rfloor \geq \lfloor \log_2 (\gamma x) \rfloor +2  \right\} \\
&\leq \p \{ Y_{r+1, \gamma} \leq - x \} \leq \p \{ Y_{0,\gamma} \leq  -x \} \leq e^{-x^2 /4},
\end{split}
\]
therefore
\[
I_4 \sim   
\frac{2^{(\{ \log_2 (\gamma x) \} -1)(r+1)}}{(r+1)! \, x^{r+1}}.
\]
Combining the asymptotics the theorem follows.
\end{proof}

Even in the untrimmed case our theorem refines the results (for general semistable laws) 
by Watanabe and Yamamuro \cite{WY}.
For $r=0$ according to Theorem \ref{th:Gtail} we have
\begin{equation} \label{eq:r0tail}
\p \{ Y_{0,\gamma} > x \} \sim 
\frac{2^{\{ \log_2 (\gamma x) \}}}{x}
\bigg[ 2^{-1}  +  \sum_{\ell=0}^1 2^{-\ell }
\p \left\{ Y_{0,\gamma} \! > \! x \hspace{-2pt} \left(1 - 2^{\ell - \{ \log_2 (\gamma x) \}} \hspace{-2pt} \right) \hspace{-2pt} \right\}
\hspace{-3pt} \bigg]. 
\end{equation}
By (\ref{eq:Levy-func}) we see that exactly the tail of the L\'evy measure appears, and we have
\[
\frac{\p \{ Y_{0,\gamma} > x \}}{-R_{\gamma}(x)} \! \sim \!
2^{-1}  +  \sum_{\ell=0}^1 2^{-\ell }
\p \left\{ Y_{0,\gamma} > x\left( \! 1 - 2^{\ell - \{ \log_2 (\gamma x) \}} \right) \right\}.
\]
From this we easily obtain
\[
\begin{split}
2^{-1} + 2^{-1} \p \{ Y_{0,\gamma} > 0 \} & = 
\liminf_{x \to \infty} \frac{\p \{ Y_{0,\gamma} > x \}}{-R_{\gamma}(x)} \\
& < \limsup_{x \to \infty} \frac{\p \{ Y_{0,\gamma} > x \}}{-R_{\gamma}(x)} = 1 + \p \{ Y_{0,\gamma} > 0 \},
\end{split}
\]
which is the statement of Theorem 2 in \cite{WY}.
From (\ref{eq:r0tail}) also follows that
\[
1=\liminf_{x \to \infty} x  \p \{ W_\gamma > x \} < 
\limsup_{x \to \infty} x \p \{ W_\gamma > x \}=2,
\]
which is Theorem 3 (i) in \cite{WY}. However, note that we determine the exact asymptotics of the ratio, and not only
the limsup and liminf of it.

\smallskip

If $\{ \log_2 (\gamma x) \} > \delta$ for some $\delta > 0$, then 
$x (1 - 2^{-\{ \log_2 (\gamma x) \}}) \to \infty$, and so the term corresponding to $\ell =0$ in (\ref{eq:Grtail-as})
converges to 0. While if $\{ \log_2 (\gamma x) \} < 1 -\delta$ for some $\delta > 0$, then 
$x (1 - 2^{1-\{ \log_2 (\gamma x) \}}) \to -\infty$, and so the term corresponding to $\ell =1$ in (\ref{eq:Grtail-as})
converges to 1. Thus the asymptotic has a simple form when $\gamma x$ is not close to a power of 2.
In particular for any $\delta \in (0,1/2)$ we have
\[
\lim_{x \to \infty, \delta < \{ \log_2 (\gamma x) \} < 1 - \delta}
\p \{ Y_{r,\gamma} > x \} \frac{x^{r+1}}{2^{\{ \log_2 (\gamma x) \}(r+1)}} = \frac{1}{(r+1)!} \,.
\]
Moreover, for $\gamma x = 2^m +c$
\[
\lim_{m\to \infty} (2^m + c) (1 - 2^{-\{ \log_2 (2^m + c) \}}) \to c,
\]
thus (\ref{eq:Grtail-as}) reads as
\[
\p \{ Y_{r,\gamma} > (2^m + c)/ \gamma \} \sim \frac{\gamma^{r+1} 2^{-m(r+1)}}{(r+1)!}
\left[1 + (2^{r+1} - 1) \p \{ Y_{0,\gamma} + A_{r,\gamma} > c/\gamma \} \right],
\]
as $m \to \infty$. In the untrimmed case ($r=0$) for $\gamma=1$ this gives
\[
\p \{ Y_{0,1} > 2^m + c \} \sim 2^{-m} \left[1 + \p \{ Y_{0,1} > c \} \right], \quad
\text{as } m \to \infty,
\]
which is exactly Martin-L\"of's asymptotics \cite[Theorem 4, formula (9)]{ML}.

\begin{remark}
The idea of our proof of Theorem \ref{th:trimlim} and the representation of the limit
goes back to LePage, Woodroofe and Zinn.

Let $Y, Y_1, Y_2, \ldots $ be i.i.d.~random variables from the domain of attraction of
an $\alpha$-stable law, $\alpha \in (0,2)$.
That is
\[
\p \{ |Y| > y \} = \frac{\ell(y)}{y^{\alpha}}, \ \,
\lim_{y \to \infty} \frac{\p \{ Y > y \} }{\p \{ |Y| > y \}} = p, \ \,
\lim_{y \to \infty} \frac{\p \{ Y < -y \} }{\p \{ |Y| > y \}} = q,
\]
with $p,q \in [0,1]$, $p + q =1$. Let $S_n$ denote the partial sum, and
let $a_n > 0$ and $b_n$ such that $(S_n - nb_n)/a_n$ converges in distribution to
an $\alpha$-stable law $S$. Let $|Y_{1,n}| \geq |Y_{2,n}| \geq \ldots \geq |Y_{n,n}|$
denote the monotone reordering of $|Y_1|, \ldots , |Y_n|$. LePage, Woodroofe and Zinn
\cite[Theorem 1']{LWZ} proved that
the limit has the representation
\[
S= \sum_{k=1}^\infty \left( \delta_k Z_k^{-1/\alpha} -
(p-q) \E Z_k^{-1/\alpha}  I(Z_k^{-1/\alpha}  < 1) \right),
\]
where $\delta_1, \delta_2, \ldots $ are i.i.d.~$\pm 1$ random variables with
$\p \{ \delta = 1 \} = p$, and independently of $\delta$'s $E_1, E_2, \ldots$
are i.i.d.~Exp(1) random variables and $Z_k = E_1 + \ldots + E_k$.
Moreover,
\[
\left( \frac{S_n - nb_n}{a_n},
\frac{1}{a_n} \left( |Y_{1,n}|, |Y_{2,n}|, \ldots, |Y_{n,n}| \right) \right)
\stackrel{\mathcal{D}}{\longrightarrow}
\left( S, ( Z_1^{-1/\alpha}, Z_2^{-1/\alpha}, \ldots ) \right).
\]
In case of the two-sided (symmetric) version of the St.~Petersburg game
similar results were obtained by Berkes, Horv\'ath and Schauer
\cite[Corollary 1.4]{BHS}.
\end{remark}

\section{The generalized St.~Petersburg game} \label{sect:general}

In this last section we consider some of the previous results in a more general
setup, in the case of the so-called generalized St.~Petersburg game.
Since the proofs are similar to the proofs in the classical case, we omit them.

In this setup Peter tosses a possibly biased coin, where the probability
of heads at each throw is $p=1-q$, and Paul's winning is
$q^{-k/\alpha}$, if the first heads appears on the $k^{\text{th}}$ toss,
where $k\in\N = \{1,2,\ldots\}$, and $\alpha > 0$ is a payoff
parameter. The classical St.~Petersburg game
corresponds to $\alpha = 1$ and $p = 1/2$.
If $X$ denotes Paul's winning in this St.~Petersburg$(\alpha, p)$ game, then
$\p\left\{ X = q^{-k/\alpha} \right\}  = q^{k-1} p$, $k\in\N$.
In this section $X_1, X_2, \ldots$ are i.i.d.~St.~Petersburg$(\alpha, p)$ random
variables, and $S_n$, $X_n^*$, and $S_{n,r}$ stands for the partial sum, partial maximum,
and the $r$-trimmed sum, respectively.

For $\alpha \geq 2$ the 
generalized St.~Petersburg distribution belongs to the domain of attraction of
the normal law. 

\smallskip

For general $\alpha, p$ we do not have a closed formula for the probabilities
$\p \{ S_n > x \}$. Nevertheless, it turns out that the generalized St.~Petersburg
distributions are not subexponential for any choice of the parameters.

\begin{lemma} \label{lemma:subexp}
Let $\alpha > 0$.
Let $X_1, X_2$ be independent St.~Petersburg$(\alpha, p)$ random variables. Then
\[
2 = \liminf_{x \to \infty} \frac{\p \{ X_1 + X_2 > x \} }{\p \{ X_1 > x \}  } <
\limsup_{x \to \infty} \frac{\p \{ X_1 + X_2 > x \} }{\p \{ X_1 > x \}  } = 2 q^{-1}.
\]
\end{lemma}

The liminf result is a consequence of a recent result by
Foss and Korshunov \cite{FK}, as they proved that
for any heavy-tailed distribution the liminf is 2.
The proof is simple, so we omit it.

By the definition of subexponential distributions in (\ref{eq:subexp-def})
the consequence of the lemma is that there is no subexponential generalized
St.~Petersburg random variable.

The tail behavior of $S_{n,r}$ in the general setup is the
following. The proof is identical to the proof in the
classical case.

\begin{theorem}
Let $\alpha > 0$. For any $n > r$
\[
\begin{split}
\p \left\{ S_{n,r} >  x \right\} & \sim 
\binom{n}{r+1} \frac{q^{ - (r+1) \{ \log_{q^{-1}} x^\alpha \}}}{x^{(r+1)\alpha}} \\
& \phantom{\sim} \, \times 
\left(1 + (q^{-r-1} -1) \p \{ S_{n-r-1} > x (1 - q^{\{ \log_{q^{-1}} x^\alpha \} /\alpha}) \} \right).
\end{split}
\]
\end{theorem}

\noindent \textbf{Acknowledgement.}
Berkes's research was supported by the grants FWF P24302-N18 and OTKA K108615.
Kevei's research was funded by a postdoctoral fellowship
of the Alexander von Humboldt Foundation.

\end{document}